\documentclass{amsart}
\usepackage{graphicx}
\usepackage{natbib}
\usepackage{hyperref}                 
\newtheorem{theorem}{Theorem}[section]
\newtheorem{example}{Example}[section]

\newtheorem{lemma}[theorem]{Lemma}
\newtheorem{proposition}[theorem]{Proposition}
\theoremstyle{definition}

\theoremstyle{remark}

\numberwithin{equation}{section}

\begin{document}
\title{A new method for fast computing unbiased estimators of cumulants
}
\author{E. Di Nardo}%
\address{Dipartimento di Matematica e Informatica, Universit\`a
        degli Studi della Basilicata, Viale dell'Ateneo Lucano 10, 85100, Potenza,
        Italy}%
\email{elvira.dinardo@unibas.it}%
\author{G. Guarino}%
\address{Medical School, Universit\`a Cattolica del Sacro Cuore
(Rome branch), Largo Agostino Gemelli 8, I-00168, Roma, Italy.}%
\email{giuseppe.guarino@rete.basilicata.it}%
\author{D. Senato}%
\address{Dipartimento di Matematica e Informatica,
        Universit\`a degli Studi della Basilicata, Viale dell'Ateneo Lucano 10, 85100, Potenza, Italy.}%
\email{domenico.senato@unibas.it}%
\keywords{univariate and multivariate $k$-statistics, univariate
and multivariate polykays, umbral calculus. \\
{\tt AMS-Primary:} 65C60, 05A40, {\tt AMS-Secondary: } 68W30, 62H99}%

\begin{abstract}
We propose new algorithms for generating $k$-statistics,
multivariate $k$-statistics, polykays and multivariate polykays.
The resulting computational times are very fast compared with
procedures existing in the literature. Such speeding up is
obtained by means of a symbolic method arising from the classical
umbral calculus. The classical umbral calculus is a light syntax
that involves only elementary rules to managing sequences of
numbers or polynomials. The cornerstone of the procedures here
introduced is the connection between cumulants of a random
variable and a suitable compound Poisson random variable. Such a
connection holds also for multivariate random variables.
\end{abstract}
\maketitle
\section{Introduction}
In the last decades, symbolic methods have been successfully used
in different mathematical areas \citep{Grossman, Wang}. Symbolic
techniques have been recently used in problems arising from
computational statistics (see \citet{Andrews2},
\citet{Zeilrberg}). The papers \citep{Dinardo2} and
\citep{Dinardo4} lie within this field. In these papers, the
theory of $k$-statistics and polykays was completely rewritten,
carrying out a unifying framework for these estimators, both in
the univariate and multivariate cases.
\par
This subject goes back to Fisher \citeyearpar{Fisher} and up to
today it was treated by means of different languages. Main
references are \citet{Stuart}, \citet{Speed1, Speed2},
\citet{McCullagh}. A more accurate list of references can be found
in \citet{Dinardo4}.
\par
The umbral techniques, investigated in \citet{Dinardo4}, have
allowed us to implement a single algorithm for $k$-statistics,
multivariate $k$-statistics, polykays and multivariate polykays
\citep{Dinardo3}. Nevertheless, the elegance of a unifying outlook
pays a price in computational costs that become comparable with
those of {\tt MathStatica} \citep{MathStatica} for polykays and
not competitive for univariate and multivariate $k$-statistics.
\par
By developing the ideas introduced in \cite{Dinardo4} for
$k$-statistics, in this paper we introduce radically innovative
procedures for generating all these estimators, by realizing a
substantial improvement of computational times compared with those
in the literature.
\par
A frequently asked question is: why are these calculations
relevant? Usually higher order objects require enormous amounts of
data to estimate with any accuracy. Nevertheless, there are
different areas, such as astronomy, astrophysics and biophysics,
which need to compute high order $k$-statistics in order to
recognizing a gaussian population \citep{Ferreira}. Undoubtedly,
an enjoyable challenge is to have efficient procedures to deal
with the involved huge amount of algebraic and symbolic
computations.
\par
The algorithms proposed here are based on the umbral language
introduced by \citet{SIAM}. Applications of the classical umbral
calculus are given in \cite{Zeilrberg1} $\div$
(\citeyear{Zeilrberg4}), where generating functions are computed
for many difficult problems dealing with counting combinatorial
objects. Applications to bilinear generating functions for
polynomial sequences are given in \cite{Gessel}.
\par
\citeauthor{Dinardo1} (\citeyear{Dinardo}) have developed the
classical umbral calculus (\citeyear{SIAM}), with special care to
probabilistic aspects. The basic device is the representation of a
unital sequence of numbers or polynomials by a symbol $\alpha,$
called an {\it umbra}. The umbra $\alpha$ is related to these
unital sequences via an operator $E$ that resembles the
expectation operator of random variables (r.v.'s). This symbolic
method provides a light syntax for handling cumulants and
factorial moments \citep{Dinardo1}, $k$-statistics then come in
hand since these are the unique symmetric unbiased estimators of
cumulants \citep{Dinardo2}. These estimators are expressed in
terms of power sums in the variables of the random sample.
\par
After recalling basic notions of the umbral language, in Section 3
we show that any cumulant can be evaluated via cumulants of a
suitable umbra. In probabilistic terms, cumulants of a r.v. can be
obtained via cumulants of a suitable compound Poisson r.v. This
link allows us the significant speed up of the algorithms for
$k$-statistics and polykays. For multivariate cumulants, the basic
tool is given by umbrae indexed by multisets. In Section 4, we
recall this symbolic device, introduced and largely used in
\citet{Dinardo4}. In Section 5, we show the connection between
multivariate cumulants of an umbra and a suitable multivariate
compound Poisson r.v. We then summarize the algorithms for
generating multivariate $k$-statistics and polykays. Finally, we
compare the computational times of the algorithms proposed here
with those of {\tt MathStatica} \citep{MathStatica} and those of
\citet{Andrews2}. Note that {\tt MathStatica} does not have a
procedure for multivariate polykays. All programs have been
executed on a PC Pentium(R)4 Intel(R), CPU 2.08 Ghz, 512MB Ram
with {\tt Maple} version 10.0 and {\tt Mathematica} version 4.2.
We choose the {\tt Maple} language due to its acknowledged
plainness in translating symbolic computations.
\section{Background to umbral calculus}
This section is aimed to recalling notation and terminology useful
to handle umbrae. More details and  technicalities can be found in
\citet{Dinardo, Dinardo1}.
\par
Formally, an umbral calculus is a syntax consisting of the
following data:
\begin{enumerate}
\item[{\it i)}] a set $A=\{\alpha,\beta, \ldots \},$ called the
{\it alphabet}, whose elements are named {\it umbrae}; \item[{\it
ii)}] a commutative integral domain $R$ whose quotient field is of
characteristic zero\footnote{For statistical applications, $R$ is
the field of real numbers.}; \item[{\it iii)}] a linear functional
$E,$ called the {\it evaluation}, defined on the polynomial ring
$R[A]$ and taking values in $R$ such that
\begin{enumerate}
\item[{\it a)}] $E[1]=1;$ \item[{\it b)}] $E[\alpha^i \beta^j
\cdots \gamma^k] = E[\alpha^i]E[\beta^j] \cdots E[\gamma^k]$ for
any set of distinct umbrae in $A$ and for $i,j,\ldots,k$
nonnegative integers ({\it uncorrelation property});
\end{enumerate} \item[{\it iv)}] an element $\varepsilon \in A,$
called the {\it augmentation}, such that $E[\varepsilon^n] = 0$
for every $n \geq 1;$ \item[{\it v)}] an element $u \in A,$ called
the {\it unity} umbra, such that $E[u^n]=1,$ for every $n \geq 1.$
\end{enumerate}
An {\it umbral polynomial} is a polynomial $p \in R[A].$ The
support of $p$ is the set of all umbrae occurring in $p.$ If $p$
and $q$ are two umbral polynomials, $p$ and $q$ are {\it
uncorrelated} if and only if their supports are disjoint. The
umbral polynomials $p$ and $q$ are {\it umbrally equivalent} if
and only if
$$E[p]=E[q], \quad \hbox{in symbols}\,\, p \simeq q.$$
\par
The {\it moments} of an umbra $\alpha$ are the elements $a_n \in
R$ such that
$$E[\alpha^n]=a_n, \,\, \forall \, n \geq 0$$
and we say that the umbra $\alpha$ {\it represents} the sequence
of moments $1,a_1,a_2,\ldots.$
\par
It is possible that two distinct umbrae represent the same
sequence of moments, in this case they are called {\it similar
umbrae}. More formally, two umbrae  $\alpha$ and $\gamma$ are said
to be similar when
$$E[\alpha^n]=E[\gamma^n] \,\,\, \forall \, n \geq
0, \quad \hbox{in symbols} \,\,  \alpha \equiv \gamma.$$ In
addition, given a sequence $1,a_1,a_2,\ldots$ in $R$ there are
infinitely many distinct and thus similar umbrae representing the
sequence.
\par
The {\it factorial moments} of an umbra $\alpha$ are the elements
$a_{(n)} \in R$ corresponding to umbral polynomials $(\alpha)_n =
\alpha (\alpha -1) \cdots (\alpha-n+1),$ for each $n \geq 1$ via
the evaluation $E,$ that is $E[(\alpha)_n]=a_{(n)}.$
\par
Two more special umbrae have been defined in the alphabet $A$: the {\it singleton}
umbra $\chi$ and the {\it Bell} umbra $\beta.$
\par
The singleton umbra $\chi$ is the umbra whose moments are all
zero, except the first $E[\chi]=1.$ As shown in \citet{Dinardo1},
its factorial moments are
\begin{equation}
E[(\chi)_n] = x_{(n)} = (-1)^{n-1}(n-1)!, \quad \forall \, n \geq
1. \label{(chifact)}
\end{equation}
The Bell umbra $\beta$ is the umbra whose factorial
moments are all equal to $1,$
\begin{equation}
E[(\beta)_n] = b_{(n)} = 1, \quad \forall \, n \geq 1.
\label{(bellfact)}
\end{equation}
Its moments are the Bell numbers. The umbra $\beta$ is therefore
the umbral counterpart of a Poisson r.v. with parameter $1$
\citep{Dinardo}.
\par
Thanks to the notion of similar umbrae, it is possible to extend
the alphabet $A$ with the so-called {\it auxiliary umbrae},
obtained via operations among similar umbrae. As a consequence, a
{\it saturated umbral calculus} can be constructed where auxiliary
umbrae are treated as elements of the alphabet \citep{SIAM}. Let
$\{\alpha_1,\alpha_2,\ldots,\alpha_n\}$ be a set of $n$
uncorrelated umbrae, similar to an umbra $\alpha.$ The symbol
$n.\alpha$ denotes an auxiliary umbra similar to the sum
$\alpha_1+\alpha_2+\cdots+\alpha_n$ and called the {\it dot
product} between the integer $n$ and the umbra $\alpha.$ Powers of
$n.\alpha$ are umbrally equivalent to the following umbral
polynomials \citep{Dinardo4}:
\begin{equation}
(n.\alpha)^i \simeq \sum_{\lambda \vdash i}(n)_{\nu_{\lambda}}
d_{\lambda} \alpha_{\lambda}, \label{(momdot1)}
\end{equation}
where the sum is over all partitions \footnote{Recall that a
partition of an integer $i$ is a sequence
$\lambda=(\lambda_1,\lambda_2,\ldots,\lambda_t),$ where
$\lambda_j$ are weakly decreasing integers and $\sum_{j=1}^t
\lambda_j = i.$ The integers $\lambda_j$ are named {\it parts} of
$\lambda.$ By the symbol $\nu_{\lambda}$ we denote the {\it
length} of $\lambda,$ that is  the number of its parts. A
different notation is $\lambda=(1^{r_1},2^{r_2},\ldots),$ where
$r_j$ is the number of parts of $\lambda$ equal to $j$ and $r_1 +
r_2 + \cdots = \nu_{\lambda}.$ We use the classical notation
$\lambda \vdash i$ to denoting \lq\lq$\lambda$ is a partition of
$i$\rq\rq.} $\lambda = (1^{r_1},2^{r_2},\ldots)$ of the integer
$i,$ $(n)_{\nu_{\lambda}} = 0$ for $\nu_{\lambda} > n$ and
\begin{equation} d_{\lambda}  = \frac{i!}{r_1!r_2!\cdots} \,
\frac{1}{(1!)^{r_1}(2!)^{r_2}\cdots} \quad \hbox{and} \quad
\alpha_{\lambda} \equiv (\alpha_{j_1})^{.r_1}
(\alpha_{j_2}^2)^{.r_2} \cdots,  \label{(not)}
\end{equation}
with $\{j_i\}$ distinct integers chosen in $\{1,2,\ldots,n\}.$ By
evaluating equivalence (\ref{(momdot1)}) via the linear functional
$E,$ we have
\begin{equation}
E[(n.\alpha)^i]=  \sum_{\lambda \vdash i}(n)_{\nu_{\lambda}}
d_{\lambda} a_{\lambda}, \label{(momdot)}
\end{equation}
where $a_{\lambda} = a_1^{r_1} \, a_2^{r_2} \, \cdots.$ Note that
if $\lambda=(1^{r_1}, 2^{r_2}, \ldots)$ is a partition  of the
integer $r,$ $\eta=(1^{s_1}, 2^{s_2}, \ldots)$ is a partition of
the integer $s$ and $\lambda + \eta = (1^{r_1 + s_1}, 2^{r_2 +
s_2}, \ldots),$ then
$$a_{\lambda + \eta} = a_{\lambda}a_{\eta} \quad \hbox{and} \quad \alpha_{\lambda + \eta} \simeq
\alpha_{\lambda}\alpha_{\eta}.$$ Properties of the auxiliary umbra
$n.\alpha$ have been extensively described in \citet{Dinardo} and
these will be recalled whenever it is necessary. It is interesting
to remark that $\alpha.n \equiv n \alpha,$ as proved in
\citet{Dinardo1}, in agreement with the meaning of the dot
product.
\par
A feature of the classical umbral calculus is the construction of
new auxiliary umbrae by suitable symbolic replacements. For
example, if we replace the integer $n$ in $n.\alpha$ by an umbra
$\gamma,$ equivalence (\ref{(momdot1)}) gives
\begin{equation}
(\gamma.\alpha)^i \simeq \sum_{\lambda \vdash i}
(\gamma)_{\nu_{\lambda}} d_{\lambda} \alpha_{\lambda}.
\label{(ombdot1)}
\end{equation}
Equivalence (\ref{(ombdot1)}) has been formally proved by using
the notion of generating function of an umbra, for further details
see \citep{Dinardo}. Note that, contrary to what happens with
$n.\alpha,$ in the dot product $\alpha.n$ the substitution of $n$
with an umbra $\gamma$ does not inherit the symbolic expression of
moments. As it is straightforward to show via (\ref{(ombdot1)}),
the dot product is therefore not commutative. This circumstance
justifies also the falling off of the right distributive law in
the dot product, so that
$$(\alpha + \delta).\gamma \equiv \alpha.\gamma + \delta.\gamma \quad \hbox{whereas}
\quad \gamma.(\alpha + \delta) \not \equiv \gamma.\alpha +
\gamma.\delta,$$ for $\alpha, \gamma, \delta \in A.$ Actually, by
considering the parallelism with the r.v.'s theory, the dot
product $\gamma.\alpha$ corresponds to a random sum, and the right
distributive law falls off similarly to what it happens for random
sums.
\par
In the dot product $\gamma.\alpha,$ by replacing the umbra
$\gamma$ by  the umbra $\gamma.\beta,$ we obtain the so-called
{\it composition} umbra of $\alpha$ and $\gamma,$ that is
$\gamma.\beta.\alpha,$ whose powers are
\begin{equation}
(\gamma.\beta.\alpha)^i \simeq \sum_{\lambda \vdash i}
\gamma^{\nu_{\lambda}} d_{\lambda} \alpha_{\lambda}.
\label{(comp)}
\end{equation}
The {\it compositional inverse} of an umbra $\alpha$ is the umbra
$\alpha^{<-1>}$ satisfying
$$\alpha^{<-1>}.\beta.\alpha \equiv \alpha.\beta.\alpha^{<-1>} \equiv
\chi.$$ In the following examples, some other fundamental
auxiliary umbrae are characterized by means of equivalence
(\ref{(ombdot1)}). The properties we are going to recall are
proved in \citet{Dinardo1, Dinardo2}.
\begin{example} The $\alpha$-partition umbra.
{\rm If $\beta$ is the Bell umbra, the umbra $\beta.\alpha$  is
called the $\alpha$-partition umbra. By taking into account
(\ref{(ombdot1)}) and (\ref{(bellfact)}), its powers are
\begin{equation}
(\beta.\alpha)^i \simeq \sum_{\lambda \vdash i} d_{\lambda}
\alpha_{\lambda}. \label{(alfapart)} \end{equation} By equivalence
(\ref{(alfapart)}), we have
\begin{equation}
\beta.u^{<-1>} \equiv \chi, \qquad \beta.\chi \equiv u, \label{(set)}
\end{equation}
where $u^{<-1>}$ denotes the compositional inverse of $u.$ The umbra
$\beta.\alpha$ corresponds to a compound Poisson r.v. of parameter $1.$}
\end{example}
\begin{example} \label{Ex1} The $\alpha$-cumulant umbra.
{\rm If $\chi$ is the singleton umbra, the umbra $\chi.\alpha$ is
called the $\alpha$-cumulant umbra. By virtue of equivalence
(\ref{(ombdot1)}), its powers are
\begin{equation}
(\chi.\alpha)^i \simeq \sum_{\lambda \vdash i} x_{(\nu_{\lambda})}
\, d_{\lambda} \, \alpha_{\lambda}, \label{(kstat)}
\end{equation} where
$x_{(\nu_{\lambda})}$ are the factorial moments (\ref{(chifact)})
of the umbra $\chi$. By taking into account (\ref{(kstat)}), the
following equivalences follow
$$\chi.\beta \equiv u, \qquad \chi.\chi \equiv u^{<-1>}.$$
Equivalence (\ref{(kstat)}) recalls the well-known expression of
cumulants $\kappa_1,\kappa_2,\ldots$ in terms of moments
$a_1,a_2,\ldots$ of a r.v.
\begin{equation}
\kappa_i=\sum_{\lambda \vdash i} (-1)^{\nu_{\lambda}-1} (\nu_{\lambda}-1)!
\, d_{\lambda} \, a_{\lambda}. \label{(conn)}
\end{equation}
The moments of the $\alpha$-cumulant umbra $\chi.\alpha$ are
therefore  called {\it cumulants} of the umbra $\alpha.$}
\end{example}
\begin{example} The $\alpha$-factorial umbra.
{\rm The umbra $\alpha.\chi$ is called the $\alpha$-factorial
umbra and its moments are the factorial moments of $\alpha,$ since
the following equivalence holds
$$(\alpha.\chi)^i \simeq (\alpha)_i.$$}
\end{example}

The commutative integral domain $R$ may be replaced by a
polynomial ring in any number of indeterminates, having
coefficient in a field $K$ of characteristic zero. Suppose
therefore to replace $R$ with the polynomial ring $K[y],$ where
$y$ is an indeterminate. The uncorrelation property  {\it iii)}
must be rewritten as
$$E[1]=1; \quad E[ y^j \alpha^k \beta^l \cdots] = y^j
E[\alpha^k]E[\beta^l] \cdots $$ for any set of distinct umbrae in
$A,$ and nonnegative integers $j,k,l,\ldots.$ In $K[y][A]$ an
umbra is said to be a {\it scalar umbra} when its moments are
elements of $K,$ while it is said to be a {\it polynomial umbra}
if its moments are polynomials of $K[y].$ A sequence of
polynomials ${\rm p}_0, {\rm p}_1, \ldots \in K[y]$ is umbrally
represented by a polynomial umbra if and only if ${\rm p}_0 = 1$
and ${\rm p}_n$ is of degree $n$ for every nonnegative integer
$n.$ If we replace the integer $n$ by the indeterminate $y$ in
(\ref{(momdot)}), then
\begin{equation}
E[(y.\alpha)^i]=  \sum_{\lambda \vdash i}(y)_{\nu_{\lambda}}
d_{\lambda} a_{\lambda}, \label{(momdoty)}
\end{equation}
where $(y)_{\nu_{\lambda}}$ denotes the lower factorial
polynomials in $K[y]$.
\par
Some other auxiliary umbrae will be used in the following. The
symbol $\alpha^{.n}$ is an auxiliary umbra denoting the product
$\alpha_1 \, \alpha_2 \, \cdots \, \alpha_n,$ where $\{\alpha_1,
\alpha_2, \ldots, \alpha_n\}$ are similar but uncorrelated umbrae.
Moments of $\alpha^{.n}$ can be easily recovered from its
definition. Indeed, if the umbra $\alpha$ represents the sequence
$1, a_1, a_2, \ldots,$ then
$$E[(\alpha^{.n})^k]=a_k^n,$$
for nonnegative integers $k$ and $n.$  The umbra $\gamma$ is said
to be {\it multiplicative inverse} of the umbra $\alpha$ if and
only if $\alpha \gamma \equiv u.$ Recall that, in dealing with a
saturated umbral calculus, the multiplicative inverse of an umbra
is not unique, but any two multiplicative inverses of the same
umbra are similar. From the definition, it follows:
$$a_n g_n = 1 \quad \forall n=0,1,2,\ldots \quad \hbox{that is } \quad g_n = \frac{1}{a_n},$$
where $a_n$ and $g_n$ are moments of $\alpha$ and $\gamma$
respectively. The multiplicative inverse of an umbra $\alpha$
should be denoted by $\alpha^{.(-1)},$ but in order to simplify
the notation and in agreement with our intuition, in the following
we will use the symbol $1/\alpha.$
\section{$k$-statistics via compound Poisson r.v.'s}
In this section we resume previous results of the authors
\citep{Dinardo4}, useful to simplify the subsequent reading.
\par
 The $i$-th $k$-statistic
$k_i$ is the unique symmetric unbiased estimator of the cumulant
$\kappa_i$ of a given statistical distribution \citep{Stuart},
that is $E[k_i] = \kappa_i.$ By virtue of (\ref{(conn)}), in
umbral terms we shall write
$$k_i \simeq (\chi.\alpha)^i.$$
In \citet{Dinardo4}, $k$-statistics have been related to cumulants
of compound Poisson r.v.'s by the following theorem.
\begin{theorem} \label{th1}
If $c_i(y)=E[(n.\chi.y.\beta.\alpha)^i], i=1,2,\ldots$ then
\begin{equation}
(\chi.\alpha)^i \simeq c_i \left( \frac{\chi.\chi}{n.\chi}
\right). \label{(3)}
\end{equation}
\end{theorem}
The statement of Theorem \ref{th1} requires some more remarks. As
stated in \citet{Dinardo1}, the umbra
$$(\chi.y.\beta).\alpha \equiv \chi.(y.\beta.\alpha)$$
is the cumulant umbra of a polynomial $\alpha$-partition umbra,
the latter corresponding to a compound Poisson r.v. of parameter
$y.$ The polynomial umbra
$$n.(\chi.y.\beta).\alpha \equiv n.\chi.(y.\beta.\alpha),$$
is therefore the sum of $n$ uncorrelated cumulant umbrae of a
polynomial $\alpha$-partition umbra. Thus, Theorem \ref{th1}
states that cumulants of $\alpha$ can be recovered from the
moments of $n.(\chi.y.\beta).\alpha,$ by a suitable replacement of
the indeterminate $y.$
\par
Usually $k$-statistics are expressed in terms of the $r$-th powers
of the data points $S_r = \sum_{i=1}^n X_i^r.$ In order to recover
this expression for $k$-statistics in umbral terms, it is
sufficient to express the polynomials $c_i(y)$ in terms of power
sums $n.\alpha^r \equiv \alpha_1^r + \cdots + \alpha_n^r.$ To this
aim, the starting point is to express the moments of a generic
umbra such as $n.(\gamma \alpha),$ with $\gamma \in A,$ in terms
of $r$-th power sums $n.\alpha^r .$
\begin{theorem} \label{thm2}
If $\alpha, \gamma \in A,$ then
\begin{equation}
[n.(\gamma\alpha)]^i \simeq \sum_{\lambda \vdash i} d_{\lambda}
(\chi.\gamma)_{\lambda} (n.\alpha)^{r_1} (n.\alpha^2)^{r_2} \cdots
\label{(eee)}
\end{equation}
with $\lambda = (1^{r_1},2^{r_2},\ldots).$
\end{theorem}
Equivalence  (\ref{(eee)}) is the way for expressing the
polynomials $c_i(y),$ umbrally equivalent to moments of
$n.\chi.y.\beta.\alpha,$ in terms of $r$-th power sums
$n.\alpha^r.$ Indeed, as
$$n.\chi.y.\beta.\alpha \equiv n.[(\chi.y.\beta)\alpha]$$
(see (31) in \citealt{Dinardo1}), we can use equivalence
(\ref{(eee)}), with $\gamma$ replaced by  $(\chi.y.\beta).$ This
is the starting point to prove the following result, by which the
fast algorithm for $k$-statistics can be easily recovered.
\begin{theorem} \label{thm3}
In $K[y],$ let
\begin{equation}
p_n(y)=\sum_{k=1}^n (-1)^{k-1} (k-1)! \, S(n,k) \, y^k,
\label{(pum)}
\end{equation}
where $S(n,k)$ are the Stirling numbers of second type. For every
$\alpha \in A$ we have
\begin{equation}
(\chi.\alpha)^i \simeq \sum_{\lambda \vdash i} d_{\lambda}
p_{\lambda} \left( \frac{\chi.\chi}{n.\chi} \right)
(n.\alpha)^{r_1} (n.\alpha^2)^{r_2} \cdots \label{(7)}
\end{equation}
with $\lambda = (1^{r_1},2^{r_2},\ldots)$ and
$p_{\lambda}(y)=[p_1(y)]^{r_1} [p_2(y)]^{r_2} \cdots.$
\end{theorem}
For the proofs of Theorems \ref{thm2} and \ref{thm3} see
\citep{Dinardo4}.
\subsection{Polykays via compound Poisson r.v.'s}
The symmetric statistic  $k_{r, \ldots,\, t}$ satisfying
$$E[k_{r, \ldots,\, t}]=\kappa_r \cdots \kappa_t,$$
where $\kappa_r, \ldots, \kappa_t$ are cumulants, generalizes
$k$-statistics and these were originally called {\it generalized
$k$-statistics} by \citet{Dressel}. Later they were called {\it
polykays} by \citet{Tukey}.
\par
As a product of uncorrelated cumulants, the umbral expression for
a polykay is simply
\begin{equation}
k_{r, \ldots,\, t} \simeq (\chi.\alpha)^r \cdots
(\chi^{\prime}.\alpha^{\prime})^t,  \label{(poll1)}
\end{equation}
with $\chi, \ldots, \chi^{\prime}$ being uncorrelated singleton
umbrae, and $\alpha, \ldots, \alpha^{\prime}$ uncorrelated umbrae
satisfying  $\alpha \equiv \cdots \equiv \alpha^{\prime}.$
\par
Also polykays are usually expressed in terms of power sums. In
\citet{Dinardo4}, starting from (\ref{(poll1)}), we have given a
compressed  umbral formula in order to express polykays in terms
of power sums. Such a formula has been implemented in {\tt Maple}
and the resulting computational times have been presented and
discussed in \citet{Dinardo3}. Despite the compressed expression
for this umbral formula, the computational cost of the resulting
algorithm involves the Bell numbers and so increases too rapidly
with $r+\cdots+t.$ A different umbral formula may be constructed
by generalizing the results of the previous section. Such a
formula is not quite expressible in a compressed form, but speeds
up the algorithm for building polykays (see the computational
times in Table 1 in Section 6).
\par
For plainness, in the following we just deal with two subindexes
$k_{r,t},$ the generalization to more than two being
straightforward.
\par
Let us consider a polynomial umbra whose moments are all equal to
$y,$ up to an integer $k,$ after which the moments are all zero.
Let us therefore define the umbra $\delta_{y,k}$ satisfying
$$\delta_{y,k} \simeq \left\{ \begin{array}{cl}
(\chi.y.\beta)^i & i=0,1,2,\ldots,k, \\
0 & i > k.
\end{array} \right.$$
\begin{lemma} Let $r,t$ be two nonnegative integers. If $k=\max\{r,t\},$ then
\begin{equation}
[n.(\delta_{y,k} \, \alpha)]^{r+t} \simeq \sum_{(\lambda \vdash r, \eta \vdash t)} y^{\nu_{\lambda}
+ \nu_{\eta}} (n)_{\nu_{\lambda} + \nu_{\eta}} d_{\lambda + \eta} \alpha_{\lambda + \eta}.
\label{(exx)}
\end{equation}
\end{lemma}
\begin{proof}
By equivalence (\ref{(momdot)}), we have
$$[n.(\delta_{y,k} \, \alpha)]^{r+t} \simeq \sum_{\xi \vdash (r+t)} (n)_{\nu_{\xi}} d_{\xi}
(\delta_{y,k})_{\xi} \alpha_{\xi},$$ since $\delta_{y,k}$ and
$\alpha$ are uncorrelated. The result follows by observing that
$(\delta_{y,k})_{\xi} \simeq 0$ for any $\xi$ satisfying
$\lambda+\eta < \xi,$ where $\lambda \vdash r, \eta \vdash t$ and
$<$ represents the lexicography order on integer partitions.
\end{proof}
\par
Let $p_{r,t}(y)$ be the polynomial obtained by evaluating the
right hand side of (\ref{(exx)}), that is
$$p_{r,t}(y) = \sum_{(\lambda \vdash r, \eta \vdash t)} y^{\nu_{\lambda}+\nu_{\eta}}
(n)_{\nu_{\lambda} + \nu_{\eta}} d_{\lambda + \eta} a_{\lambda +
\eta}.$$ The proof of the following theorem is straightforward and
allows us to express products of uncorrelated cumulants by using
the polynomials $p_{r,t}(y).$
\begin{theorem} \label{thm4}
If $q_{r,t}$ is the umbral polynomial obtained via $p_{r,t}(y)$ by
replacing $y^{\nu_{\lambda}+\nu_{\eta}}$ by
\begin{equation}
\frac{(\chi.\chi)^{\nu_{\lambda}}(\chi^{\prime}.\chi^{\prime})^{\nu_{\eta}}}{(n.\chi)^{\nu_{\lambda}+\nu_{\eta}}}
\frac{d_{\lambda} d_{\eta}}{d_{\lambda + \eta}},
\label{(subs)}
\end{equation}
then
$$(\chi.\alpha)^r (\chi^{\prime}.\alpha^{\prime})^t \simeq q_{r,t}.$$
\end{theorem}
These two last results are sufficient to express the polykay
$k_{r,t}$ in terms of power sums. The steps are summarized in the
following:
\begin{enumerate}
\item[{\it i)}] we apply Theorem \ref{thm2} to the polynomial
umbra $n.(\delta_{y,k} \, \alpha),$ with $k =\max\{r,t\},$ in
order to link the polynomials $p_{r,t}(y)$ to power sums, that is
\begin{equation}
[n.(\delta_{y,k} \, \alpha)]^{r+t} \simeq  p_{r,t}(y) \simeq \sum_{\xi \vdash (r+t)} d_{\xi}
(\chi.\delta_{y,k})_{\xi} (n.\alpha)^{s_1} (n.\alpha^2)^{s_2} \cdots;
\label{(ee1)}
\end{equation}
\item[{\it ii)}]  we evaluate the cumulants of the umbra
$\delta_{y,k}$ by means of (\ref{(kstat)}), by recalling that
moments corresponding to powers greater than $k$ are zero;
\item[{\it iii)}] we replace occurrences of
$y^{\nu_{\lambda}+\nu_{\eta}}$ in (\ref{(ee1)}) by (\ref{(subs)}),
thanks to Theorem \ref{thm4}.
\end{enumerate}
The steps {\it i)} --  {\it iii)} are the building blocks of the
fast algorithm for generating polykays.
\section{The multivariate case: umbrae indexed by multisets}
In order to consider multivariate $k$-statistics and polykays, we
need of the notion of multivariate moments and multivariate
cumulants of an umbral monomial. The umbral tools necessary to
deal with multivariate moments are introduced in \citet{Dinardo4}.
Here we recall basic notation and equivalences in order to
generalize Theorems \ref{thm2} and \ref{thm4} to the multivariate
case.
\par
A {\it multiset} $M$ is a pair $(\bar{M}, f),$ where $\bar{M}$ is
a set, called the {\it support} of the multiset, and $f$ is a
function from $\bar{M}$ to nonnegative integers. For each $\mu \in
\bar{M},$ $f(\mu)$ is the {\it multiplicity} of $\mu.$ The {\it
length} of the multiset $(\bar{M},f),$ usually denoted by $|M|,$
is the sum of multiplicities of all elements of $\bar{M},$ that is
$$|M| = \sum_{\mu \in \bar{M}} f(\mu).$$ When the support of $M$
is a finite set, say $\bar{M}=\{\mu_1, \mu_2, \ldots, \mu_k\},$ we
will write
$$M = \{\mu_1^{(f(\mu_1))},\mu_2^{(f(\mu_2))}, \ldots, \mu_k^{(f
(\mu_k))}\} \quad \hbox{or} \quad M = \{\underbrace{\mu_1, \ldots,
\mu_1}_{f(\mu_1)}, \ldots, \underbrace{\mu_k, \ldots,
\mu_k}_{f(\mu_k)}\}.$$ For example the multiset
$$M =
\{\underbrace{\alpha, \ldots, \alpha}_{i}\} = \{\alpha^{(i)}\}$$
has length $i,$ support $\bar{M} = \left\{ \alpha \right\}$ and
$f(\alpha)=i.$ Recall that a multiset $M_i=(\bar{M_i},f_i)$ is a
{\it submultiset} of $M=(\bar {M}, f)$ if $\bar{M_i} \subseteq
\bar{M}$ and $f_i(\mu) \leq f(\mu), \, \forall \mu \in \bar{M_i}.$
\\
\begin{example}
{\rm If $M=\{ \alpha, \alpha, \gamma, \delta, \delta\}$ then
$M_1=\{\alpha,\alpha\}$ is a submultiset with support $\bar{M_1} = \{\alpha\}$ and
$f_1(\alpha)=2.$ Also $M_2=\{\alpha, \delta, \delta\}$ is a  submultiset with
support $\bar{M_2} = \{\alpha, \delta\}$ and $f_2(\alpha)=1, f_2(\delta)=2.$}
\end{example}
\smallskip \par In the following we set
\begin{equation}
\, \, \mu_{M} =  \prod_{\mu \in \bar{M}} \mu^{f(\mu)}, \quad
(n.\mu)_{M} =  \prod_{\mu \in \bar{M}} (n.\mu)^{f(\mu)}, \quad
[n.(\chi \, \mu)]_{M} =  \prod_{\mu \in \bar{M}} [n.(\chi
\mu)]^{f(\mu)}.  \label{(mmomb4)}
\end{equation}
For instance, if $M = \{\alpha^{(i)}\}$ then $\alpha_{M}=
\alpha^i,$ $(n.\alpha)_M = (n.\alpha)^i,$ $[n.(\chi \alpha)]_M =
[n.(\chi \alpha)]^i.$

A {\it subdivision} of a multiset $M$ is a multiset
$S=(\bar{S},g)$ of $k \leq |M|$ non empty submultisets
$M_i=(\bar{M}_i, f_i)$ of $M$ satisfying
\begin{enumerate}
\item[{\it i)}] $\cup_{i=1}^k \bar{M}_i = \bar{M};$ \item[{\it
ii)}] $\sum_{i=1}^k f_i(\mu) = f(\mu)$ for any $\mu \in \bar{M}.$
\end{enumerate}
\medskip \par
\begin{example}
{\rm  Multisets $S_1=\{\{\alpha,\gamma\},\{\alpha\}, \{\delta,\delta\}\}$
and $S_2=\{\{\alpha,\gamma,\delta\},\{\alpha,\delta\}\}$ are subdivisions of
$M=\{ \alpha, \alpha, \gamma, \delta, \delta\}.$}
\end{example}
\smallskip \par
By extending the notation (\ref{(mmomb4)}), we set
\begin{equation}
\mu_{S} =  \prod_{M_i \in \bar{S}} \mu_{M_i}^{g(M_i)}, \quad
(n.\mu)_{S} =  \prod_{M_i \in \bar{S}} (n.\mu_{M_i})^{g(M_i)},
\label{(momb6bis)}
\end{equation}
\begin{equation}
[n.(\chi \mu)]_{S} =  \prod_{M_i \in \bar{S}} [n.(\chi \mu_{M_i})]^{g(M_i)} .
\label{(momb6)}
\end{equation}
\begin{example} \label{ex1}
{\rm If $M=\{ \mu_1, \mu_1, \mu_2\}$, then
$S_1=\{\{\mu_1\},\{\mu_1,\mu_2\}\}$ is a subdivision of $M.$ The
support of $S_1$ consists of two multisets, $M_{1}=\{\mu_1\}$ and
$M_{2}=\{\mu_1, \mu_2\},$ each of one with multiplicity $1,$
therefore $[n.(\chi\mu)]_{S_{1}} =  n.(\chi \mu_{M_1})
n.(\chi\mu_{M_2}).$ Since $n.(\chi \mu_{M_1}) =  n.(\chi \mu_1)$
and $n.(\chi\mu_{M_2}) = n.(\chi \mu_1 \mu_2),$ we have
$[n.(\chi\mu)]_{S_{1}}= n.(\chi \mu_1) n.(\chi \mu_1 \mu_2).$}
\end{example}
\medskip \par
We may construct a subdivision of the multiset $M$ by a suitable
set partition. Recall that a partition $\pi$ of a set $C$ is a
collection $\pi=\{B_1, B_2, \ldots, B_k\}$ with $k \leq n$
disjoint and non-empty subsets of $C$ whose union is $C.$ We
denote by $\Pi_n$ the set of all partitions of $C.$ Suppose the
elements of $M$ to be all distinct, build a set partition and then
replace each element in any block by the original one. By this
way, any subdivision corresponds to a set partition $\pi$ and we
will write $S_{\pi}.$ Note that $|S_{\pi}|=|\pi|$ and it could be
$S_{\pi_1} = S_{\pi_2}$ for $\pi_1 \ne \pi_2,$ as the following
example shows.
\medskip \par
\begin{example}{\rm If  $M=\{ \alpha, \alpha, \gamma, \delta, \delta\}$ suppose
to label each element of $M$ in such a way $C=\{ \alpha_1,
\alpha_2, \gamma_1, \delta_1, \delta_2\}.$ The subdivision
$S_1=\{\{\alpha,\gamma\},\{\alpha\}, \{\delta,\delta\}\}$
corresponds to the partition $\pi_1=
\{\{\alpha_1,\gamma_1\},\{\alpha_2\}, \{\delta_1,\delta_2\}\}$ of
$C.$ We have $|S_{1}|=|\pi_1|.$ Note that the subdivision $S_1$
also corresponds to the partition $\pi_2=
\{\{\alpha_2,\gamma_1\},\{\alpha_1\}, \{\delta_1,\delta_2\}\}.$}
\end{example}
\medskip \par
In the following, we denote by $n_{\pi}$ the number of set
partitions in $\Pi_{|M|}$ corresponding to the same subdivision
$S$ of the multiset $M.$
\par
If $M=\{\alpha^{(i)}\},$ then subdivisions are of type
\begin{equation}
S=\{\underbrace{\{\alpha\},\ldots,\{\alpha\}}_{r_1},
 \underbrace{\{\alpha^{(2)}\},\ldots,\{\alpha^{(2)}\}}_{r_2},
 \ldots\},
\label{submult}
\end{equation}
with $r_1 + 2 \,r_2 + \cdots = i.$ The support of $S$ is
$\bar{S}=\{\{\alpha\}, \{\alpha^{(2)}\}, \ldots\},$ so that
(\ref{(momb6bis)}) and (\ref{(momb6)})  give
\begin{equation}
(n.\alpha)_S \equiv (n.\alpha)^{r_1} (n.\alpha^2)^{r_2} \cdots \quad [n.(\chi\alpha)]_S \equiv
[n.(\chi \alpha)]^{r_1} [n.(\chi \alpha^2)]^{r_2} \cdots.
\label{(mmob7)}
\end{equation}
Before ending this summary, we recall one more notation. Suppose
$S$ is a subdivision of the multiset $M$ of type
$S=\{M_1^{(g(M_1))},M_2^{(g(M_2))}, \ldots, M_j^{(g(M_j))}\}.$ By
the symbol $\mu^{.S}$ we denote
\begin{equation}
\mu^{.S} \equiv  (\mu_{M_1})^{.g(M_1)} \cdots
(\mu^{\prime}_{M_j})^{.g(M_j)}, \label{(1.2)}
\end{equation}
where $\mu_{M_t}$ are uncorrelated umbral monomials. Observe that
also $\mu^{.S}$ is a multiplicative function, that is  if $S_1$
and $S_2$ are subdivisions of $M$ then
$$\mu^{.(S_1 + S_2)} \equiv \mu^{.S_1} \mu^{.S_2},$$
where $S_1 + S_2$ denotes the disjoint union of $S_1$ and $S_2.$
If $M=\{\alpha^{(i)}\},$ then
\begin{equation}
\alpha_{\lambda} \equiv \alpha^{.S}
\label{(eqq)}
\end{equation}
with $\lambda=(1^{r_1}, 2^{r_2}, \ldots).$ The notation
(\ref{(1.2)}) cames in handy in order to evaluate umbral
polynomials like $(n.\mu)_M$ in terms of moments of the umbral
monomials running in $M.$ Indeed, observe that a different way to
write equivalence (\ref{(momdot1)}) follows by using subdivisions
$S_{\pi}$ of $M=\{\alpha^{(i)}\},$ that is
\begin{equation}
(n.\alpha)^i \simeq \sum_{\pi \in \Pi_i} (n.\chi)^{|S_{\pi}|} \alpha^{.S_{\pi}}.
\label{(2m)}
\end{equation}
Replace the multiset $M=\{ \alpha^{(i)}\}$ by a generic multiset
$M,$ then it follows
\begin{equation}
(n.\mu)_M \simeq \sum_{\pi \in \Pi_i} (n.\chi)^{|S_{\pi}|} \mu^{.S_{\pi}}.
\label{(3m)}
\end{equation}
We end the section by adding one more remark. Let $N$ be a
submultiset of $M.$ By the symbol $(\mu_M)_N$ we denote the
monomial umbra $\mu_{M \cap N}.$ This notation allows us to
generalize equivalence (\ref{(3m)}) to umbral polynomial
$(n.\mu_M)_N$ with $N \subset M,$ that is
\begin{equation}
(n.\mu_M)_N \simeq \sum_{\pi \in \Pi_{|N|}} (n.\chi)^{|S_{\pi}|} (\mu_M)^{.S_{\pi}}.
\label{(4mm)}
\end{equation}
\section{Multivariate $k$-statistics via compound Poisson r.v.'s}
In \citet{Dinardo4}, multivariate moments and multivariate
cumulants of an umbral monomial are introduced. Let $M =
\{\mu_1^{(f(\mu_1))}, \mu_2^{(f(\mu_2))}, \ldots,
\mu_r^{(f(\mu_r))}\}$ be a multiset of length $i.$ A multivariate
moment is the element of $K[y]$ corresponding to the umbral
monomial $\mu_{M}$ via evaluation $E,$ that is
$$
E[\mu_M] = m_{t_1 \ldots \, t_r},
$$
where $t_j=f(\mu_j)$ for $j=1,2,\cdots,r.$ The corresponding
multivariate cumulant is the element of $K[y]$ satisfying
\begin{equation}
E[(\chi.\mu)_M]= \kappa_{t_1 \ldots t_r}. \label{momc}
\end{equation}
For example, if $M = \{\alpha^{(i)}\}$ then $(\chi.\mu)_M \simeq
(\chi.\alpha)^i.$ As for $k$-statistics, the umbra
$(\chi.y.\beta.\mu)_M$ is the cornerstone for building efficiently
multivariate $k$-statistics. Let us observe that the evaluation of
$(\chi.y.\beta.\mu)_M$ gives the umbral counterpart of joint
cumulants of a multivariate compound Poisson r.v. with parameter
$y.$ We need therefore to characterize the evaluation of
$[n.(\chi.y.\beta.\mu)]_M.$
\begin{proposition}  \label{thm7}
Let $M$ be  a multiset of length $i.$ The umbra
$[n.(\chi.y.\beta.\mu)]_M$ is umbrally equivalent to the umbral
polynomial
\begin{equation}
c_M(y) = \sum_{\pi \in \Pi_i} (n.\chi)^{|S_{\pi}|} y^{|S_\pi|} \mu^{.S_\pi},
\label{(1m)}
\end{equation}
where $S_\pi$ are the subdivisions of the multiset $M$ corresponding to
the partitions $\pi.$
\end{proposition}
\begin{proof}
In  equivalence (\ref{(3m)}), replace the generic umbral monomial $\mu$ by
$\chi.y.\beta.\mu.$ We have
\begin{equation}
[n.(\chi.y.\beta.\mu)]_M \simeq \sum_{\pi \in \Pi_i} (n.\chi)^{|S_{\pi}|} [(\chi.y.\beta)\mu]^{.S_{\pi}},
\label{(4m)}
\end{equation}
where the form on the right-hand side is worked out by means of
the equivalence $\chi.y.\beta.\mu \equiv (\chi.y.\beta)\mu.$ The
umbrae $(\chi.y.\beta)$ and $\mu$ are uncorrelated, so that
$$[(\chi.y.\beta)\mu]^{.S_{\pi}} \equiv (\chi.y.\beta)^{.S_{\pi}} \mu^{.S_{\pi}}.$$
Let $S_{\pi}=\{M_1^{(g(M_1))},M_2^{(g(M_2))}, \ldots,
M_j^{(g(M_j))}\}.$ Equivalence (\ref{(1m)}) follows from
(\ref{(4m)}), by observing that
$$(\chi.y.\beta)^{.S_{\pi}} \simeq y^{|S_{\pi}|},$$
since the umbra $\chi.y.\beta$ has moments all equal to $y,$ and $\sum g(M_i) = |S_{\pi}|.$
\end{proof}
\begin{theorem} \label{multt}
If $c_M(y)$ are the umbral polynomials given in (\ref{(1m)}), then
\begin{equation}
c_M \left( \frac{\chi.\chi}{n.\chi} \right) \simeq (\chi.\mu)_M.
\end{equation}
\end{theorem}
\begin{proof}
The result follows directly from (\ref{(1m)}) by replacing $y$ by
$\frac{\chi.\chi}{n.\chi}$ and by recalling that
\begin{equation}
(\chi.\mu)_{M} \simeq  \sum_{\pi \in \Pi_i}
(\chi.\chi)^{|S_{\pi}|} \, \mu^{.{S_\pi}}, \label{(dotprod1)}
\end{equation}
an equivalence which follows from (\ref{(3m)}) by replacing $n$ by
$\chi.$
\end{proof}
\begin{theorem}
Let  $p_{{\pi}}(x)=[p_1(x)]^{r_1} [p_2(x)]^{r_2} \cdots,$ with
$p_n(x)$ given in (\ref{(pum)}) and $\pi$ a partition of
$\Pi_{|M|}$ with $r_1$ blocks of cardinality $1,$ $r_2$ blocks of
cardinality $2,$ and so on, then
\begin{equation}
(\chi.\mu)_M \simeq \sum_{\pi \in \Pi_i} p_{\pi} \left( \frac{\chi.\chi}{n.\chi} \right) (n.\mu)_{S_{\pi}}.
\label{(multv)}
\end{equation}
\end{theorem}
\begin{proof}
First observe that $[n.(\chi.y.\beta.\mu)]_M \equiv
[n.((\chi.y.\beta)\mu)]_M,$ so that
$$c_M(y) \simeq [n.((\chi.y.\beta)\mu)]_M.$$
We need to express $[n.((\chi.y.\beta)\mu)]_M$ in terms of power sums. To this aim,
note that, by using (\ref{(mmob7)}) and (\ref{(eqq)}), equivalence (\ref{(eee)})  can be rewritten as
\begin{equation}
[n.(\gamma \alpha)]_M \simeq \sum_{\pi \in \Pi_i} (\chi.\gamma)^{.S_{\pi}} (n.\alpha)_{S_{\pi}},
\label{(gen4.4)}
\end{equation}
where $S_{\pi}$ is a subdivision of the multiset $M=\{\alpha^{(i)}
\}.$ Replace $M$ by any multiset. Equivalence (\ref{(gen4.4)})
becomes
\begin{equation}
[n.(\gamma \mu)]_M \simeq \sum_{\pi \in \Pi_i} (\chi.\gamma)^{.S_{\pi}} (n.\mu)_{S_{\pi}}.
\label{(gen4.4b)}
\end{equation}
In equivalence (\ref{(gen4.4b)}), suppose to replace $\gamma$ by
$\chi.y.\beta,$ then
\begin{equation}
[n.((\chi.y.\beta)\mu)]_M \simeq \sum_{\pi \in \Pi_i} (\chi.\chi.y.\beta)^{.S_{\pi}} (n.\mu)_{S_{\pi}}.
\label{(gen4.4mu)}
\end{equation}
Let $S_{\pi} = \{M_1^{(g(M_1))},M_2^{(g(M_2))}, \ldots,
M_j^{(g(M_j))}\},$ then
$$(\chi.\chi.y.\beta)^{.S_{\pi}} \equiv [(\chi.\chi.y.\beta)_{M_1}]^{.g(M_1)} \cdots
[(\chi^{\prime}.\chi^{\prime}.y.\beta^{\prime})_{M_j}]^{.g(M_j)}.$$
Observe that
$$(\chi.\chi.y.\beta)_{M_i} = \prod_{\mu \in \bar{M}_i} (\chi.\chi.y.\beta)^{f(\mu)} = (\chi.\chi.y.\beta)^{|M_i|},$$
so that
$$E[(\chi.\chi.y.\beta)^{.S_{\pi}}] = [p_{\scriptscriptstyle{|M_1|}}(y)]^{g(M_1)} \cdots
[p_{\scriptscriptstyle{|M_j|}}(y)]^{g(M_j)}=p_{\pi}(y)$$ if
$S_{\pi}$ is the subdivision corresponding to the partition $\pi.$
By replacing $y$ by $(\chi.\chi)/(n.\chi),$ equivalence
(\ref{(multv)}) is proved.
\end{proof}
\par
Recall that the multivariate $k$-statistics are the unique
symmetric unbiased estimators of joint cumulants. Since these
estimators are umbrally equivalent to $(\chi.\mu)_M,$ with a
suitable choice of the multiset $M,$ the expression for
multivariate $k$-statistics in terms of power sums is given by the
right-hand side of equivalence (\ref{(multv)}).
\subsection{Multivariate polykays via compound Poisson r.v.'s}
The symmetric statistic  $k_{t_1 \ldots t_r; \ldots;\, l_1 \ldots
l_m}$ satisfying
$$E[k_{t_1 \ldots t_r; \ldots;\, l_1 \ldots l_m}]=\kappa_{t_1 \ldots t_r} \cdots
\kappa_{l_1 \ldots l_m},$$ where $\kappa_{t_1 \ldots t_r}, \ldots,
\kappa_{l_1 \ldots l_m}$ are multivariate cumulants, generalizes
polykays. As product of uncorrelated multivariate cumulants, the
umbral expression for a multivariate polykay is simply
\begin{equation}
k_{t_1 \ldots t_r; \ldots;\, l_1 \ldots l_m} \simeq (\chi.\mu)_T
\cdots (\chi^{\prime}.\mu^{\prime})_L,  \label{(pollm1)}
\end{equation}
with $\chi, \ldots, \chi^{\prime}$ being uncorrelated singleton
umbrae and $T, \ldots, L$ multisets of umbral monomials such that
$$T=\{\mu_1^{(t_1)},\ldots, \mu_r^{(t_r)}\}, \, \ldots \,, L =
\{{\mu}_1^{(l_1)},\ldots, {\mu}_m^{(l_m)}\}.$$ Also for
multivariate polykays we have given a compressed umbral formula in
terms of multivariate power sums \citep{Dinardo4}. Such a formula
has been implemented in {\tt Maple} and the resulting
computational times have been presented and discussed in
\citet{Dinardo3}. Here we generalize the procedure given for
univariate polykays, speeding up the algorithm.
\par
For plainness, in the following we just deal with two multisets
$T$ and $L,$ the generalization being straightforward.
\par
Let $N$ be the disjoint union of all submultisets respectively of
$T$ and $L$ and suppose to denote by $+$ the disjoint union of two
multisets.
\medskip \par
\begin{example}
{\rm Let $T=\{\mu_1,\mu_2\}$ and $L=\{\mu_1\}.$ The disjoint union
of all submultisets respectively of $T$ and $L$ is
$N=\{\{\mu_1,\mu_2\}, \{\mu_1\},\{\mu_2\},\{\mu_1\}\}.$ If
$T=\{\mu_1,\mu_2\}$ and $L=\{\mu_3\}$ then $N=\{ \{\mu_1,\mu_2\},
\{\mu_1\}, \{\mu_2\},  \{\mu_3\}\}.$}
\end{example}
\medskip \par
As before, we need of a polynomial umbra, indexed by a suitable
multiset, which behaves as a filter on subdivisions of $T + L,$ by
deleting those which are not disjoint unions of subdivisions
respectively of $T$ and $L.$ Suppose therefore $M_i=(\bar{M}_i,
g),$ a submultiset of $T + L.$ Let us define the umbra
$\delta_{y,N}$ satisfying
\begin{equation}
(\delta_{y,N})_{M_i} \simeq \left\{ \begin{array}{ll}
0 & \hbox{if $M_i \not \subset N$}, \\
(\chi.y.\beta)_{M_i} \simeq (\chi.y.\beta)^{|M_i|} \simeq y & \hbox{otherwise}. \\
\end{array} \right.
\label{(defu)}
\end{equation}
Let $S_{\nu}$ be a subdivision of $T + L,$ and $S_{\pi}$ and
$S_{\tau}$ subdivisions respectively of $T$ and $L,$ without
taking into account the distinct labels. Via (\ref{(defu)}), the
following equivalence results
\begin{equation}
(\delta_{y,N})^{.S_{\nu}} \simeq \left\{ \begin{array}{ll}
0 & \hbox{if} \, S_{\nu} \not < S_{\pi} + S_{\tau}, \\
(\chi.y.\beta)^{.S_{\nu}} \simeq y^{|S_{\nu}|} & \hbox{otherwise}, \\
\end{array} \right.
\label{(defu1)}
\end{equation}
where $<$ denotes the natural extension to subdivisions of the
refinement relation defined on the lattice of set partitions.
\begin{lemma}
If $\delta_{y,N}$ is the umbra defined in (\ref{(defu)}), then
\begin{equation}
[n.(\delta_{y,N} \, \mu)]_{T+L} \simeq \sum_{(\pi \in \Pi_{|T|}, \tau \in \Pi_{|L|})}
(n.\chi)^{|S_{\pi}|+|S_{\tau}|} y^{|S_{\pi}|+|S_{\tau}|} \, \mu^{.(S_{\pi}+S_{\tau})}.
\label{(exx1)}
\end{equation}
\end{lemma}
\begin{proof}
From (\ref{(4mm)}), we have
\begin{equation}
[n.(\delta_{y,N} \, \mu)]_{T+L} \simeq \sum_{\nu \in \Pi_{|T+L|}}
(n.\chi)^{|S_{\nu}|} \, \delta_{y,N}^{.S_{\nu}} \,\,\,
\mu^{.S_{\nu}}, \label{(exx3)}
\end{equation}
since $\delta_{y,N}$ is uncorrelated with any element of $T+L.$
Due to (\ref{(defu1)}), in the sum on the right hand side of
(\ref{(exx3)}), the addends which give a non-zero contribution
are only those corresponding to subdivisions which can be split in
a subdivision of $T$ and a subdivision of $L,$ that is
\begin{equation}
[n.(\delta_{y,N} \, \mu)]_{T+L} \simeq \sum_{(\pi \in \Pi_{|T|},
\tau \in \Pi_{|L|})} (n.\chi)^{|S_{\pi}|+|S_{\tau}|} \,\,
\delta_{y,N}^{.S_{\pi}} \,\, \delta_{y,N}^{.S_{\tau}} \,\,
\mu^{.S_{\pi}}
 \, \mu^{.S_{\tau}}.
\label{(exx4)}
\end{equation}
Observing that $\delta_{y,N}^{.S_{\pi}} \simeq y^{|S_{\pi}|}$ and
$\delta_{y,N}^{.S_{\tau}} \simeq y^{|S_{\tau}|},$ the result
follows immediately since $\mu^{.S_{\pi}} \mu^{.S_{\tau}} \simeq
\mu^{.(S_{\pi}+S_{\tau})},$  due to (\ref{(defu1)}).
\end{proof}
\par
By recalling that subdivisions corresponding to different
partitions can be equal, equivalence (\ref{(exx1)}) may be
rewritten as
\begin{equation}
[n.(\delta_{y,N} \, \mu)]_{T+L} \simeq \sum_{(S_{\pi}, S_{\tau})}
n_{\pi + \tau} \,\, (n.\chi)^{|S_{\pi}|+|S_{\tau}|} \,\,
y^{|S_{\pi}|+|S_{\tau}|} \,\, \mu^{.(S_{\pi}+S_{\tau})},
\label{(exx5)}
\end{equation}
where $n_{\pi + \tau}$ is the number of set partition pairs
$(\pi,\tau)$ corresponding to subdivision $S_{\pi} + S_{\tau}.$
Assume
\begin{equation} p_{T,L}(y) = \sum_{(S_{\pi}, S_{\tau})} n_{\pi +
\tau} \, (n.\chi)^{|S_{\pi}|+|S_{\tau}|} y^{|S_{\pi}|+|S_{\tau}|}
\, \mu^{.(S_{\pi}+S_{\tau})}. \label{polmul}
\end{equation}
Thanks to (\ref{(exx5)}), the next theorem is proved by simple
calculations and allows us to express products of uncorrelated
multivariate cumulants by using the polynomials $p_{T,L}(y).$
\begin{theorem} \label{thm8}
Suppose $n_{\pi}$ (respectively $n_{\tau}$) the number of set
partitions in $\Pi_{|T|}$ (respectively $\Pi_{|L|}$) corresponding
to the subdivision $S_{\pi}$ (respectively $S_{\tau}$) and
$n_{\pi+\tau}$ the number of set partitions in $\Pi_{|T + L|}$
corresponding to the subdivision $S_{\pi} + S_{\tau}.$ If
$q_{T,L}$ is the umbral polynomial obtained from $p_{T,L}(y)$ by
replacing $y^{|S_{\pi}|+|S_{\tau}|}$ with
\begin{equation}
\frac{(\chi.\chi)^{|S_{\pi}|}(\chi^{\prime}.\chi^{\prime})^{|S_{\tau}|}}{(n.\chi)^{|S_{\pi}|+|S_{\tau}|}}
\frac{n_{\pi} n_{\tau}}{n_{\pi + \tau}},
\label{(subs1)}
\end{equation}
then
$$(\chi.\mu)_T \, (\chi^{\prime}.\mu^{\prime})_L \simeq q_{T,L}.$$
\end{theorem}
\begin{proof}
Due to (\ref{(dotprod1)}), product of multivariate cumulants may be written as:
$$(\chi.\mu)_T  \,\, (\chi^{\prime}.\mu^{\prime})_L \simeq \sum_{(\pi \in \Pi_{|T|}, \tau \in \Pi_{|L|})}
\,\, (\chi.\chi)^{|S_{\pi}|}
(\chi^{\prime}.\chi^{\prime})^{|S_{\tau}|} \,\,
\mu^{.(S_{\pi}+S_{\tau})}.$$ The previous equivalence can be
rewritten as
\begin{equation}
(\chi.\mu)_T  (\chi^{\prime}.\mu^{\prime})_L \simeq \sum_{(S_{\pi}, S_{\tau})}
n_{\pi} n_{\tau} (\chi.\chi)^{|S_{\pi}|} (\chi^{\prime}.\chi^{\prime})^{|S_{\tau}|}
\mu^{.(S_{\pi}+S_{\tau})}.
\label{(compp)}
\end{equation}
The result follows by comparing the right hand side of
(\ref{(compp)}) with $p_{T,L}(y)$ in (\ref{polmul}) where
$y^{|S_{\pi}|+|S_{\tau}|}$ has been replaced by (\ref{(subs1)}).
\end{proof}
\par
Via equivalence (\ref{(gen4.4b)}), the following equivalence holds
\begin{equation}
[n.(\delta_{y,N} \mu)]_{T+L} \simeq \sum_{\nu \in \Pi_{|T+L|}} (\chi.\delta_{y,N})^{.S_{\nu}}
(n.\mu)_{S_{\nu}},
\label{(gen4.6b)}
\end{equation}
by which it is possible to express multivariate polykays in terms
of power sums. The algorithm is summarized in the following:
\begin{enumerate}
\item[{\it i)}] by equivalence (\ref{(gen4.6b)}), we evaluate $n.(\delta_{y,N} \, \mu)$
in terms of power sums in order to link the polynomials $p_{L,T}(y)$ to power sums;
\item[{\it ii)}]  we evaluate the cumulants of the umbra $\delta_{y,N}$ by means of
$$(\chi.\delta_{y,N})_{M} \simeq  \sum_{\pi \in \Pi_i} (\chi.\chi)^{|{S_\pi}|} \,
\delta_{y,N}^{.{S_\pi}}$$ which is an obvious generalization of
equivalence (\ref{(dotprod1)}); \item[{\it iii)}] we replace
occurrences of $y^{|S_{\pi}|+|S_{\tau}|}$ in
$(\chi.\delta_{y,N})_{M}$ by (\ref{(subs1)}).
\end{enumerate}
\section{Computational comparisons}
Tables 1 and 2 show comparisons of computational times among four
different software packages. The first one, which we call {\tt AS}
algorithms, has been implemented in {\tt Mathematica} and refers
to procedures explained in \citep{Andrews2}, see
http://fisher.utstat.toronto.edu/david/SCSI/chap.3.nb. The second
one refers to the package {\tt MathStatica} \citep{MathStatica}.
Note that in this package, there are no procedures devoted to
multivariate polykays. The third package, named {\tt Fast}
algorithms, has been implemented in {\tt Maple 10.x} by using the
results of this paper. The last procedure, named {\tt Polyk}, has
been described in \citep{Dinardo3}. Let us remark that, for all
the considered procedures, the results are in the same output form
and have been performed by the authors on the same platform. To
the best of our knowledge, there is no R implementation for
$k$-statistics and polykays.
\begin{table}[htbp]
\begin{center}
\small{\caption{Comparison of computational times in sec. for
$k$-statistics and polykays. Missed computational times \lq\lq
means greater than $20$ houres\rq\rq. }}
\begin{tabular}{|c|c|c|c|c|c|} \hline
$k_{t, \ldots,\, l}$ & {\tt AS} Algorithms & {\tt MathStatica} &
{\tt Fast}-algorithms & {\tt Polyk}-algorithm \\ \hline
$k_5$    &  0.06    & 0.01    &  0.01 & 0.08 \\
$k_7$    &  0.31    & 0.02    &  0.01 & 0.03 \\
$k_9$    &  1.44    & 0.04    &  0.01 & 0.16 \\
$k_{11}$ &  8.36    & 0.14    &  0.01 & 0.23 \\
$k_{14}$ & 396.39   & 0.64    &  0.02 & 1.33 \\
$k_{16}$ & 57982.40 & 2.03    &  0.08 & 4.25 \\
$k_{18}$ &  -       & 6.90    &  0.16 & 13.70 \\
$k_{20}$ &   -      & 25.15   &  0.33 & 42.26 \\
$k_{22}$ &   -      & 81.70   &  0.80 & 172.59 \\
$k_{24}$ &   -      & 359.40  &  1.62 & 647.56 \\
$k_{26}$ &   -      & 1581.05 &  2.51 & 3906.19 \\
$k_{28}$ &   -      & 6505.45 &  4.83 & 21314.65 \\
$k_{3,2}$&  0.06    & 0.02    &  0.01 & 0.02 \\
$k_{4,4}$&  0.67    & 0.06    &  0.02 & 0.06 \\
$k_{5,3}$&  0.69    & 0.08    &  0.02 & 0.07 \\
$k_{7,5}$&  34.23   & 0.79    &  0.11 & 0.70 \\
$k_{7,7}$&  435.67  & 2.52    &  0.26 & 2.43 \\
$k_{9,9}$&    -     & 27.41   &  2.26 & 23.32 \\
$k_{10,8}$&   -     & 30.24   &  2.98 & 25.06 \\
$k_{4,4,4}$& 34.17  & 0.64    &  0.08 & 0.77 \\
\hline
\end{tabular}
\end{center}
\end{table}
\par
The {\tt Polyk} algorithm, introduced in \citet{Dinardo3}, has the
advantage to give $k$-statistics, multivariate $k$-statistics,
polykays and multivariate polykays, depending on input parameters.
That is one algorithm for the whole matter. The computational
times of {\tt Polyk} are better than those of {\tt AS} algorithms
in all cases. {\tt Polyk} works better than {\tt MathStatica} for
polykays but is not competitive for $k$-statistics. {\tt
MathStatica} has not a procedure for multivariate polykays.
\begin{table}[htbp]
\begin{center}
\caption{Comparison of computational times in sec. for
multivariate $k$-statistics and multivariate polykays. For {\tt
AS} Algorithms and {\tt Polyk}-algorithm, missed computational
times means \lq\lq greater than 20 houres\rq\rq. For {\tt
MathStatica}, missed computational times means \lq\lq procedures
not available\rq\rq.}
\begin{tabular}{|c|c|c|c|c|} \hline
$k_{t_1 \ldots\, t_r;\, l_1 \ldots l_m}$ & {\small {\tt AS}
Algorithms} & {\small {\tt MathStatica}} & {\small {\tt
Fast}-algorithms} & {\small {\tt Polyk}-algorithm}
\\ \hline
$k_{3 \, 2}$        &   0.25   &  0.03 & 0.01 & 0.03 \\
$k_{4 \, 4}$        &  28.36   &  0.16 & 0.02 & 0.34 \\
$k_{5 \, 5}$        &  259.16  &  0.55 & 0.06 & 1.83 \\
$k_{6 \, 5}$        &  959.67  &  1.01 & 0.16 & 4.61 \\
$k_{6 \, 6}$        &    -     &  2.20 & 0.28 & 12.08 \\
$k_{7 \, 6}$        &    -     &  4.01 & 0.53 & 33.22 \\
$k_{7 \, 7}$        &    -     &  8.49 & 1.04 & 95.19 \\
$k_{8 \, 6}$        &    -     &  7.37 & 1.09 & 91.80 \\
$k_{8 \, 7}$        &    -     & 14.92 & 2.19 & 300.60 \\
$k_{3 \, 3 \, 3}$   & 1180.03  & 0.88  & 0.47 & 2.90 \\
$k_{4 \, 3 \, 3}$   &    -     & 2.00  & 0.40 & 9.26 \\
$k_{4 \, 4 \, 3}$   &    -     & 4.80  & 0.94 & 34.20 \\
$k_{4 \, 4 \, 4}$   &    -     & 13.53 & 2.30 & 155.03 \\
$k_{1 \, 1;\, 1 \, 1}$ & 0.05    &   -   & 0.01 & 0.01\\
$k_{2 \, 1;\, 1 \, 1}$ & 0.20    &   -   & 0.01 & 0.03 \\
$k_{2 \, 2;\, 1 \, 1}$ & 1.22    &   -   & 0.03 & 0.05 \\
$k_{2 \, 2;\, 2 \, 1}$ & 6.30    &   -   & 0.08 & 0.09 \\
$k_{2 \, 2;\, 2 \, 2}$ & 33.75   &   -   & 0.14 & 0.30 \\
$k_{2 \, 1;\, 2 \, 1;\, 2 \, 1}$   & 78.94  &  - &  0.22 & 0.45\\
$k_{2 \, 2;\, 1 \, 1;\, 1 \, 1}$   & 30.01  &  - &  0.14 & 0.20 \\
$k_{2 \, 2;\, 2 \, 1;\, 1 \, 1}$   & 126.19 &  - &  0.28 & 0.55 \\
$k_{2 \, 2;\, 2 \, 1;\, 2 \, 1}$   & 398.42 &  - &  0.55 & 1.66 \\
$k_{2 \, 2;\, 2 \, 2;\, 1 \, 1}$   & 464.45 &  - &  0.61 & 1.59\\
$k_{2 \, 2;\, 2 \, 2;\, 2 \, 1}$   & 1387.00&  - &  1.25 & 5.52\\
$k_{2 \, 2;\, 2 \, 2;\, 2 \, 2}$   & 3787.41&  - &  2.91 & 20.75 \\
\hline
\end{tabular}
\end{center}
\end{table}
\par
Finally, from Table 1 and 2, it is evident that there is a
significant improvement of computational times realized by the
{\tt Fast}-algorithms, compared to the other three packages. The
{\tt Fast}-algorithms are available at the following web page
http://www.unibas.it/utenti/dinardo/fast.pdf.

In Table 3, we quote computational times for the $k$-statistics,
the polykays and the multivariate ones given in Tables
1 and 2, obtained with forthcoming {\tt MathStatica} release 2,
by using {\tt Mathematica} 6.0, on Mac OS X, with Mac Pro
2.8GHz (Colin Rose, private communication)
\begin{table}[htbp]
\begin{center}
\small{\caption{Computational times in sec., for the $k$-statistics,
the polykays and the multivariate ones given in Tables
1 and 2, obtained with forthcoming {\tt MathStatica} release 2.}}
\begin{tabular}{|c|c|c|c|} \hline
$k_{t, \ldots,\, l}$ & {\tt MathStatica} 2 & $k_{t_1 \ldots\, t_r;\, l_1 \ldots l_m}$ &
{\tt MathStatica} 2 \\ \hline
$k_5$    &  0.008   & $k_{3 \, 2}$  &  0.012 \\
$k_7$    &  0.017   & $k_{4 \, 4}$  &  0.009 \\
$k_9$    &  0.039   & $k_{5 \, 5}$  &  0.345 \\
$k_{11}$ &  0.084   & $k_{6 \, 5}$  &  0.592 \\
$k_{14}$ &  0.329   & $k_{6 \, 6}$  &  1.230 \\
$k_{16}$ &  0.917   & $k_{7 \, 6}$  &  2.107 \\
$k_{18}$ &  2.804   & $k_{7 \, 7}$  &  4.215 \\
$k_{20}$ &  9.363   & $k_{8 \, 6}$  &  3.595 \\
$k_{22}$ &  32.11   & $k_{8 \, 7}$  &  7.359 \\
$k_{3,2}$&  0.012   & $k_{3 \, 3 \, 3}$    &  0.529 \\
$k_{4,4}$&  0.044   & $k_{4 \, 3 \, 3}$    &  2.552 \\
$k_{7,5}$&  0.434   & $k_{4 \, 4 \, 4}$    &  6.926 \\
$k_{7,7}$&  1.288   & $k_{1 \, 1;\, 1 \, 1}$   & 0.006 \\
$k_{9,9}$&  11.89   & $k_{2 \, 1;\, 1 \, 1}$   &  0.014 \\
$k_{10,8}$& 12.39   & $k_{2 \, 2;\, 1 \, 1}$   &  0.038 \\
$k_{4,4,4}$& 0.359  & $k_{2 \, 2;\, 2 \, 1}$   &  0.085 \\
           &         & $k_{2 \, 2;\, 2 \, 2}$   & 0.020 \\
           &         & $k_{2 \, 1;\, 2 \, 1;\, 2 \, 1}$   & 0.227 \\
           &         & $k_{2 \, 2;\, 1 \, 1;\, 1 \, 1}$   & 0.154 \\
           &         & $k_{2 \, 2;\, 2 \, 1;\, 1 \, 1}$   & 0.413 \\
           &         & $k_{2 \, 2;\, 2 \, 1;\, 2 \, 1}$   & 0.928 \\
           &         & $k_{2 \, 2;\, 2 \, 2;\, 1 \, 1}$   & 1.063 \\
           &         & $k_{2 \, 2;\, 2 \, 2;\, 2 \, 1}$   & 2.622 \\
           &         & $k_{2 \, 2;\, 2 \, 2;\, 2 \, 2}$   & 6.402 \\

\hline
\end{tabular}
\end{center}
\end{table}
\bigskip \bigskip

\end{document}